\documentclass[a4paper,11pt]{article}

\usepackage[T1]{fontenc}
\usepackage[utf8]{inputenc}
\usepackage[english]{babel}
\usepackage{lmodern}  
\usepackage{microtype}
\usepackage[margin=3cm]{geometry}

\usepackage{enumitem}
\setlist[itemize]{label= --, itemsep=0pt, topsep=5pt, leftmargin=20pt}

\usepackage{amsmath,amssymb,amsthm}
\usepackage{graphicx}
\usepackage{booktabs}
\usepackage{bussproofs}

\usepackage{hyperref}
\hypersetup{colorlinks=true, linkcolor=black, citecolor=black, urlcolor=black}
\usepackage[capitalise,nameinlink]{cleveref}
\usepackage{orcidlink}

\usepackage{csquotes}
\usepackage[backend=biber,style=numeric,sorting=nyt]{biblatex}
\theoremstyle{definition}
\newtheorem{definition}{Definition}[section]
\newtheorem{remark}[definition]{Remark}

\theoremstyle{plain}
\newtheorem{theorem}[definition]{Theorem}
\newtheorem{lemma}[definition]{Lemma}
\newtheorem{proposition}[definition]{Proposition}
\newtheorem{corollary}[definition]{Corollary}

\newenvironment{keyword}
  {\par\medskip\noindent\textbf{Keywords: }\ignorespaces}
  {\par\medskip}
\newcommand{\sep}{\enspace$\cdot$\enspace}
\newcommand{\affmark}[1]{\textsuperscript{#1}}

\let\oldtitle\title
\renewcommand{\title}[1]{%
  \oldtitle{%
    \rule{\textwidth}{0.4pt}\\[1ex]
    {\LARGE #1}\\
    \rule{\textwidth}{0.4pt}
  }%
}

\newcommand{\NN}{\ensuremath{\mathbb{N}}}

\newcommand{\CD}{\mathcal{D}}
\newcommand{\CG}{\mathcal{G}}
\newcommand{\CR}{\mathcal{R}}
\newcommand{\CI}{\mathcal{I}}
\newcommand{\CP}{\mathcal{P}}
\newcommand{\CQ}{\mathcal{Q}}

\newcommand{\FF}{\mathbf{F}}
\newcommand{\TT}{\mathbf{T}}

\newcommand{\BB}{\mathbb{B}}
\newcommand{\List}{\mathbb{L}}

\newcommand{\NAO}{\mathsf{NA}^{\omega}}
\newcommand{\MAO}{\mathsf{MA}^{\omega}}
\newcommand{\HAO}{\mathsf{HA}^{\omega}}
\newcommand{\PAO}{\mathsf{PA}^{\omega}}
\newcommand{\XAO}{\mathsf{XA}^{\omega}}

\newcommand{\term}{\mathsf{Term}}
\newcommand{\form}{\mathsf{Form}}

\newcommand{\GG}{{\neg\neg}}

\newcommand{\true}{\mathsf{tt}}
\newcommand{\false}{\mathsf{ff}}
\newcommand{\Succ}{\mathsf{S}}
\newcommand{\nil}{\mathsf{nil}}

\newcommand{\Split}{\mathsf{split}}
\newcommand{\case}{\mathsf{Cases}}
\newcommand{\rec}{\mathcal{R}}
\newcommand{\atom}[1]{\mathsf{at}(#1)}

\usepackage{marvosym}
\usepackage{fontawesome5}
\newcommand{\email}[1]{\Letter\  \href{mailto:#1}{\texttt{#1}}}
\newcommand{\website}[2]{\Mundus\ \href{#1}{\texttt{#2}}}

\title{On the Limits of Recursive Characterizations in the Refined $A$-Translation}

\author{
Franziskus Wiesnet \affmark{a} \orcidlink{0000-0003-3870-6984} 
}

\begin{document}
\maketitle
\makeatletter
\begingroup
\renewcommand{\thefootnote}{}
\renewcommand{\theHfootnote}{acknowledgements}
\long\def\@makefntext#1{\noindent #1}
\footnotemark
\footnotetext{\textbf{Acknowledgements.} The research presented in this paper was funded by the Austrian Science Fund (FWF), Grant-DOI: \href{https://www.fwf.ac.at/forschungsradar/10.55776/ESP576}{10.55776\allowbreak/ESP576}.}
\endgroup
\makeatother
\setcounter{footnote}{0}
\begin{flushleft}
\footnotesize\itshape
\affmark{a} Institut für diskrete Mathematik and Geometrie, 
TU Wien, Wiedner Hauptstraße 8-10/104, Vienna, 1040,
Austria, European Union, \email{franziskus.wiesnet@tuwien.ac.at}, \website{https://www.wiesnet.eu}{wiesnet.eu}\\
\end{flushleft}

\begin{abstract}
This paper studies the limits of recursive syntactic classifications in
proof theory and program extraction, using the refined $A$-translation as
a central example. The refined $A$-translation, due to Berger, Buchholz,
and Schwichtenberg, is based on recursively defined classes of formulas in
minimal arithmetic $\mathsf{MA}^\omega$, in particular the classes of
definite and goal formulas. One of its basic properties is that
$D[\bot := F] \to D$ is intuitionistically derivable for every definite
formula $D$.

Schwichtenberg and Wainer observed that this property also holds for formulas
outside the class of definite formulas and asked for a useful characterization
of all formulas $D$ for which $D[\bot := F] \to D$ is intuitionistically
derivable. In addition to definite formulas, the refined $A$-translation
involves three further classes of formulas satisfying related
properties. We show that none of these four properties admits a recursive
characterization.

In addition to this negative result, we extend the framework of refined
$A$-translation in two directions. First, we add conjunction $\wedge$ to
the language of $\mathsf{MA}^\omega$, whose original formulation contains
only the logical connectives $\forall$ and $\to$, and adapt the formula
classes accordingly. Second, we present the corresponding slightly extended
formulation of the refined $A$-translation theorem and discuss possible
recursive extensions of these classes.

Finally, we discuss a small prover written in Rust which implements the
theory $\mathsf{MA}^\omega$ and the four formula classes. The
prover is not used as a formal verification of the results, but
serves as a case study for examining Rust as a programming language for proof
assistants. We highlight some advantages and drawbacks of Rust in this
setting, including its strong type system, support for partial constructions,
ownership and borrowing model, modularity, and testing infrastructure.
\end{abstract}

\section{Background and Overview}
In proof theory, constructive logic, and logic programming one frequently
works with syntactically defined classes of formulas that guarantee certain
proof-theoretic properties. Classical examples include 
Rasiowa--Harrop formulas in intuitionistic proof theory \cite[p.~143]{troelstra1988constructivism}, as well as
Horn clause, hereditary Harrop formulas and goal formulas in logic programming \cite{miller2021survey}. Such classes provide recursively checkable sufficient conditions and
thereby connect derivability properties with the syntactic structure of
formulas.

The present paper studies this question in the setting of the refined
$A$-translation. The refined $A$-translation belongs to the tradition of
proof translations from classical to constructive systems. Its historical
origin is Friedman's $A$-translation
\cite{friedman1978classically} from 1978, which was introduced in the study of
classically and intuitionistically provably recursive functions.  It~was also independently
discovered by Dragalin \cite{dragalin1980new}. The method was further investigated only a few years later by
Leivant \cite{leivant1985syntactic} as well as by Troelstra and van Dalen \cite{troelstra1988constructivism}.

In 2002, Berger, Buchholz, and Schwichtenberg systematized the idea of a refined $A$-translation as a tool for program extraction from classical proofs \cite{berger2002refined}. This refinement was necessary in order to apply formal program extraction to the obtained constructive proofs. This idea had already appeared in an earlier study about the greatest common divisor by Berger and Schwichtenberg \cite{berger1996greatest}.
The refined $A$-translation improves the plain $A$-translation by taking the syntactic structure of formulas into
account. In particular, it distinguishes between classes of definite
formulas and goal formulas, which are defined by recursion over the formula structure. 
The corresponding classes of formulas then each satisfy a certain property
which is needed in order to prove the main theorem on the refined
$A$-translation (Corollary \ref{cor:ATrans}).

Schwichtenberg and Wainer
\cite{schwichtenberg2012proofs} observed that these properties also hold
for formulas outside the recursively defined classes and formulated the
problem of finding useful recursive characterizations of all formulas
satisfying the corresponding properties. Our main result shows that such
recursive characterizations do not exist. More precisely, assuming the
consistency of $\PAO$, none of the proof-theoretic properties
associated with the refined $A$-translation admits a recursive
characterization.

Besides this negative result, we extend the framework of the refined
$A$-translation in several directions. In particular, we extend the
language of $\MAO$ by conjunction and adapt the corresponding formula
classes accordingly. We also discuss possible recursive extensions of the
formula classes and formulate the refined $A$-translation theorem in this
extended setting.

The refined $A$-translation has been implemented and used in several
case studies in the
Minlog proof assistant (see for example \cite{berger2002refined, seisenberger2003constructive}).
In connection with the present work, we wrote our
own prover\footnote{We refer to the software only as a ``prover'' or ``proof checker'', since, unlike more comprehensive proof assistants, it merely supports the implementation of proofs and does not, for example, (yet) provide a genuine tactic script or a parser.} for the theory $\MAO$ in the Rust programming language \cite{wiesnet2026checker}. Since $\MAO$ is a relatively manageable theory, this provided a good opportunity to do so.
The Rust prover had two main purposes. First, it allowed us to investigate
how changes in the recursive definitions of the individual classes of
formulas affect the resulting theory. Second, it allowed us to examine
how well Rust is suited as a programming language for implementing a
prover.

Since the first stable version of Rust was released in 2015, whereas
programming languages typically used for proof assistants, such as Haskell
and OCaml, date back well before the year 2000, further experience with Rust
as a programming language for proof assistants still needs to be gained. This
paper aims to contribute to this process. The Rust prover is not needed for the theorems about the $A$-translation presented in this article.
It served mainly as a tool for inspiration, and the paper is written in
such a way that the first four sections can be understood without any knowledge of programming.
However, the implementation also led to some interesting insights, which
are discussed in the last section of this article. 
\medskip 

\paragraph{Organization of this article.}\smallskip

\noindent In Section~2, we introduce the theories
$\HAO$, $\PAO$, $\NAO$, and $\MAO$, together with the basic proof-theoretic
tools needed later. \smallskip

\noindent In Section~3, we define the formula classes used in the
refined $A$-translation and prove the corresponding non-recursiveness
results. \smallskip

\noindent Section~4 contains the theorem on the refined $A$-translation in
our extended setting. Then we list some applications and discuss the impact of the negative statement. \smallskip

\noindent Finally, in Section~5, we discuss the Rust
implementation and some observations concerning Rust as a programming
language for proof assistants.
\section{Systems of Arithmetic with Finite Types}
In this section, we introduce \emph{Heyting Arithmetic} $\HAO$ with finite types, its classical counterpart \emph{Peano Arithmetic} $\PAO$ with finite types, and two fragments of $\HAO$ -- the \emph{Negative Arithmetic} $\NAO$ and the \emph{Minimal Arithmetic} $\MAO$.  The language of the two fragments is restricted to negative formulas only. In particular, there is no strong disjunction $\vee$ and no strong existential quantifier $\exists$. In contrast to the other sources in which these fragments are defined such as \cite{berger2002refined},\cite[Section 7.3]{schwichtenberg2012proofs} and \cite[Section 1.3]{trifonov2012analysis}, we include conjunction in the definition. In \cite[Section 2]{ratiu2011refinement}, conjunction was also already added to $\NAO$; however, it was not considered in the definition of definite, goal, relevant, and irrelevant formulas \cite[Definition 3.5]{ratiu2011refinement}. Thus, Definition \ref{defi:DGRI} in our paper is a slight extension of the theory. We will see later that this makes no difference in terms of logical content, but makes the formulas themselves somewhat easier to understand.

\subsection{Formula Language}
We begin by defining types and typed terms. This definition applies to each of the four theories mentioned above, therefore we just talk about ``types'' and ``terms'' without mentioning the theory explicitly.
\begin{definition}[Types]
To define types we fix a countable list of type variables, which we usually denote by $\xi$ and $\zeta$. Furthermore, $\BB$ denotes the type of booleans, $\NN$ the type of natural numbers and $\List(\tau)$ the type of lists of $\tau$. Types are then recursively defined by
\begin{align*}
\tau, \rho ::= \xi \mid \BB \mid \NN \mid \List(\tau) \mid \tau \to \rho  \mid \tau \times \rho.
\end{align*}
\end{definition}
\begin{remark}
In the above definition, we used type variables. In what follows, we will introduce further kinds of variables. Variables occur either bound or free, and free variables may be substituted. If $E$ is an expression (i.e.~a type, term, formula or proof), $x$ is a variable, and $t$ is another expression, then $E[x := t]$ denotes the substitution of all free occurrences of $x$ by $t$. In a given context, we use the following convention: if an expression is introduced as $E(x)$ and we later write $E(t)$, then we mean $E[x := t]$. Note that the notation $E(x)$ is used regardless of whether $x$ occurs freely in $E$ or not.

Since the concept of bound and free variables and substitution as well as $\alpha$-equivalence is intuitively easy to understand, we refrain from giving a formal definition here. A human-readable description of these concepts can be found, for example, in \cite[Section 1.1.2]{trifonov2012analysis}. It should be noted, however, that in particular a fully formal definition of substitution is considerably more complex than it is often presented to be, since unintended variable collisions may occur. For example, if $x$ is a variable of type $\xi$, then, when substituting $\xi$ by another type $\tau$, one has to ensure that $x$ is renamed in such a way that it does not collide with a variable that originally had type $\tau$. For a human working with pen and paper, this is quite clear; for a
computer, however, the precise implementation is more complex. For an exact definition, we refer to the corresponding part of the code of our Rust prover \cite{wiesnet2026checker}.

\end{remark}
\begin{definition}[Terms]
For each type $\tau$ we fix a countable set of object variables denoted by $x^\tau,y^\tau, z^\tau$.
Furthermore, we have the following typed constructors:
\begin{itemize}
\item the pair constructor $\langle\cdot,\cdot \rangle : \tau \to \rho \to \tau\times\rho$ for any types $\tau$ and $\rho$.
\item $\true: \BB$ and $\false : \BB$ denoting truth and falsity,
\item zero $0:\NN$ and the successor $\Succ: \NN \to \NN$,
\item the empty list $\nil_\tau: \List(\tau)$ and the list constructor $\cdot ::_\tau \cdot : \tau \to \List(\tau)\to \List(\tau)$ for any type $\tau$,
\end{itemize}
and the following typed destructors:
\begin{itemize}
\item the pair splitting $\Split_{\rho,\sigma}^\tau:\rho\times\sigma \to (\rho\to\sigma \to \tau) \to \tau$ for any types $\rho,\sigma,\tau$,
\item the if-operator $\case^\tau: \BB \to\tau\to\tau\to\tau$ for any type $\tau$,
\item the recursion for natural numbers $\rec^\tau_\NN:\NN \to \tau\to (\NN\to\tau\to\tau) \to \tau$ for any type $\tau$,
\item the recursion for lists  $\rec^\tau_{\List(\rho)}:\List(\rho) \to \tau\to (\rho\to\List(\rho)\to\tau\to\tau) \to \tau$ for any types $\tau,\rho$.
\end{itemize} 
Typed terms are then defined recursively by the following rules:
\begin{itemize}
\item Each object variable, constructor and destructor is a term with its type given above.
\item If $t : \tau\to \rho $ and $s:\tau$ are terms, then also $ts : \rho$ is a term.
\item If $x: \tau$ is an object variable and $t: \rho$ is a term, then $\lambda_xt : \tau \to \rho$ is a  term.
\end{itemize}
We denote the set of all terms by $\term$.
\end{definition}
\begin{remark}
It should be noted that there are the usual conversion rules for the terms defined above. For example $$(\lambda_xt(x))(s) \mapsto t(s)\qquad\text{and}\qquad \case^\tau\,\true\,t\,s \mapsto t.$$ For our purposes, however, these are not needed.  We therefore omit a detailed discussion and refer the reader to \cite[Section 1.2.1]{trifonov2012analysis} or \cite[Section 2.1]{ratiu2011refinement}. The same applies to the proof terms from Definition~\ref{def:ProofTerms}.
\end{remark}
\begin{definition}[Formulas]
Formulas in $\NAO$ are defined by:
\begin{itemize}
\item For each boolean term $t:\BB$ the atomic formula $\atom{t}$ is a formula.
\item If $A$ and $B$ are formulas, then so are $A\to B$ and $A\wedge B$.
\item If $A$ is a formula and $x$ is an object variable, then $\forall_x A$ is a formula. 
\end{itemize}
As abbreviations, we define $\TT := \atom\true$ and $\FF := \atom\false$. Negation $\neg$ is defined by $\neg A := A\to \FF$.\\
To define formulas in $\MAO$ we fix a new symbol $\bot$ and add the clause that $\bot$ is also a formula.

To define formulas in $\HAO$ and $\PAO$ we add the following rules to $\NAO$:
\begin{itemize}
\item If $A$ and $B$ are formulas, then so is $A\vee B$.
\item If $A$ is a formula and $x$ is an object variable, then $\exists_xA$ is a formula.
\end{itemize}
For $\XAO\in\{\NAO,\MAO,\HAO,\PAO\}$ we denote the sets of all formulas in $\XAO$ by $\form_{\XAO}$. As $\form_{\HAO}=\form_{\PAO}$ we often denote them just by $\form$.
\end{definition}

Note that disjunction and the existential quantifier do not occur in their strong form. However, the weak forms $\tilde{\vee}$ and $\tilde{\exists}$ can be defined by 
\begin{align*}
A\tilde\vee B &:= \neg (\neg A \wedge \neg B),\\
\tilde{\exists}_xA &:= \neg\forall_x\neg A
\end{align*} A weak form of conjunction could also be defined by $$A\tilde{\wedge} B := \neg(A\to \neg B).$$
Furthermore, there are no predicate variables in $\NAO$ and in $\MAO$ only $\bot$ functions as a predicate variable. 
This functionality is shown in the following definition:
\begin{definition}
For formulas $S$ and $A$ we define the substitution $A^S$ of $\bot$ by $S$ as follows:
\begin{align*}
\bot^S &:= S,\\
\atom t^S &:= \atom t,\\
(A\circ B)^S & := A^S \circ B^S \text{ for } \circ\in\{\wedge, \to\},\\
(\forall_x A)^S & := \forall_x A^S.
\end{align*}
To avoid collisions with free variables in $S$ one might have to rename $x$ in the last rule.
\end{definition}

To conclude the description of the theories, we now define formal proofs. These are based on Gentzen's calculus of natural deduction \cite{gentzen1935untersuchungen}. Here we use the compact representation by proof terms.

\begin{definition}[Proof terms]
\label{def:ProofTerms}
For each formula $A$ we fix a countable set of assumption variables usually denoted by $u^A,v^A,w^A,\cdots$.
Furthermore we have the following axioms in $\NAO$:
\begin{itemize}
\item $\mathsf{Truth} : \TT$
\item $\mathcal{C}^{b,A} : \forall_b(A(\true) \to A(\false) \to A(b))$
\item $\mathsf{Ind}_\NN^{n,A}: \forall_n(A(0) \to \forall_n(A(n)\to A(\Succ n)) \to A(n))$
\item $\mathsf{Ind}_{\List(\rho)}^{l,A} : \forall_l(A(\nil)\to \forall_{x,l}(A(l) \to A(x::l)) \to A(l))$
\end{itemize}
Proof terms are then defined as follows:
\begin{itemize}
\item Each assumption variable is a proof of its formula.
\item Each axiom is a proof of its formula.
\item If $M$ is a proof of $A$ and $N$ is a proof of $B$, then $(M,N)$ is a proof of $A\wedge B$.
\item If $M$ is a proof of $A\wedge B$, then $\pi_0M$ is a proof of $A$ and $\pi_1M$ is a proof of $B$.
\item If $M$ is a proof of $A\to B$ and $N$ is a proof of $A$, then $MN$ is a proof of $B$.
\item If $M$ is a proof of $B$ and $u$ is an assumption variable of $A$, then $\lambda_uM$ is a proof of $A\to B$.
\item If $M$ is a proof of $\forall_xA(x)$ and $t$ is a term with the same type as $x$, then $Mt$ is a proof of $A(t)$.
\item If $M$ is a proof of $A$ and $x$ is a variable which is not free in any free assumption of $M$, then $\lambda_x M$ is a proof of $\forall_x A$.
\end{itemize}
\medskip
To define proofs in $\MAO$ we add the axiom
\begin{align*}
\bot^+: \FF \to\bot
\end{align*}
to the axioms above. The rules for proof terms in $\MAO$ are then the same as the rules for proof terms in $\NAO$.
\medskip\\
To define proofs in $\HAO$ we take the axioms and rules of $\NAO$ and add the following axioms:
\begin{itemize}
\item $\vee_0^+: A\to A\vee B$
\item $\vee_1^+: B\to A\vee B$
\item $\vee^-: A\vee B \to (A\to C) \to (B\to C) \to C$
\item $\exists^+: A(t) \to \exists_xA(x)$, where $t$ is a term and $x$ is a variable with the same type.
\item $\exists^- : \exists_xA \to \forall_x(A \to C) \to C$, where $x$ is a variable which is not free in $C$.
\end{itemize}
\medskip The theory $\PAO$ is then $\HAO$ together with the law of excluded middle
\begin{align*}
\mathsf{lem}: A\vee \neg A
\end{align*}
added as axiom.
\end{definition}

For $\XAO\in \{\NAO,\MAO,\HAO,\PAO\}$ we write $\XAO\vdash A$ if there is a proof term of $A$ without free assumption variables.  Furthermore, we write $\XAO \vdash A\leftrightarrow B$ for $\XAO \vdash (A\to B) \wedge (B\to A)$ or equivalently $\XAO\vdash A\to B$ and $\XAO \vdash B\to A$.

\begin{remark}
In contrast to the proof terms defined above, proofs are usually represented in practice as trees. This, however, is merely a matter of presentation, and the two representations are equivalent, as can be seen, for example, from \cite[Table 2.1]{ratiu2011refinement}. Since the representation by proof terms is somewhat more compact, we present the proofs in this form here.

In the course of this paper, we will repeatedly prove derivability statements, that is, statements of the form $\XAO \vdash A$. However, since our proofs are intended to be understandable for human readers, we will give proof terms only in simple cases and otherwise use natural language. From the above definition, however, we can see quite clearly which arguments may be used and how. Nevertheless, some statements are also proved by induction on a proof. In this case, we naturally consider each axiom and each rule one by one.
\end{remark}

\subsection{Useful Lemmas}
The following two lemmas will be used throughout the article:
\begin{lemma}[Ex-falso]
\label{lem:efq}
For $\XAO \in \{\NAO,\MAO,\HAO,\PAO\}$, if A is a formula of $\mathsf{XA}^\omega$, then $\FF\to A$ is provable in $\mathsf{XA}^\omega$.
\end{lemma}
\begin{proof}
The proof is by induction on $A$. \\
The case $A = \bot$ is only necessary in $\MAO$ and follows directly from $\bot^+$. \\
If $A=\atom{t}$ for some $t:\BB$, we apply $t$ to $\mathcal{C}^{b,\atom b}$ and get $\TT \to \FF \to \atom t$. As $\TT$ holds by $\mathsf{Truth}$, we exactly get $\FF \to \atom t$.\\
If $A=B\to C$, we get $\FF \to A$ directly from the induction hypothesis $\FF\to C$. \\
If $A = B\wedge C$, we get $\FF \to A$ directly from the induction hypotheses $\FF\to B$ and $\FF\to C$.\\
If $A=\forall_xB$, we get a proof of $\FF \to B$ without any free assumption (and therefore no free assumptions on $x$). Hence, we get a proof of $\FF\to\forall_xB$.\\
For $\HAO$ and $\PAO$ we finally consider $A=B \vee C$ and $A=\exists_xB$. By the induction hypothesis, we have in both cases $\FF\to B$. Using $\vee_0^+$ we get $\FF \to B\vee C$, and by $\exists^+$ we get $\FF \to\exists_xB$.
\end{proof}

\begin{lemma}
\label{lem:substbot}
For $\XAO\in\{\HAO,\MAO,\NAO\}$, let $S$ be a formula of  $\XAO$, $A$ a formula in $\MAO$ and assume that there is a proof of $A$ in $\MAO$ with free assumptions $(u_1:A_1),\dots, (u_n:A_n)$, then there is a proof of $A^S$ in $\XAO$  with  free assumptions $(\tilde{u}_1:A_1^S),\dots, (\tilde{u}_n:A_n^S)$, where $\tilde{u}_1,\dots,\tilde{u}_n$ are new assumption variables.
\end{lemma}
\begin{proof}
First, we rename all variables occurring in $S$ to fresh variables, so that no variable clashes can occur in what follows. After the construction, we substitute the free variables in $S$ back by the original variables. We may therefore assume that $S$ contains only variables which occur neither in $A$ nor in the derivation of $A$.
\medskip\\
We proceed by induction on the derivation $M$ of $A$.

If $M = u : A$ is an assumption variable, we replace it by a new assumption variable $\tilde{u}:A^S$.

If $M=\bot^+$, then $A=\FF\to \bot$. In this case, a proof of
$
A^S=\FF\to S
$
is given by Lemma~\ref{lem:efq}.

If $M$ is any other axiom, that is, if
\[
M=\mathcal{C}^{b,B}, \qquad
M=\mathsf{Ind}_\NN^{n,B},
\qquad\text{or}\qquad
M=\mathsf{Ind}_{\List(\rho)}^{l,B}
\]
for some formula $B$, then we simply replace $B$ by $B^S$. By the remark at the beginning of the proof, no variable collisions occur
in this process.

If $M=\lambda_x N$, then $A=\forall_x B$, and $N$ is a derivation of $B$. Moreover, $x$ is not free in any free assumption of $N$. By the induction hypothesis, we obtain a derivation $\tilde{N}$ of $B^S$. As explained at the beginning of the proof, we may assume that $x$ is not free in $S$. Since $x$ is also not free in any free assumption of $N$, we may assume that $x$ is not free in any free assumption of $\tilde{N}$. Hence,
$
\lambda_x \tilde{N}
$
is the desired derivation.

If $M=Nt$ for some term $t$, then there is some formula $B=B(x)$ with $A=B(t)$, and $N$ is a derivation of $\forall_x B$. By the induction hypothesis, we obtain a derivation $\tilde{N}$ of
$
\forall_x B^S.
$
By the remark at the beginning of this proof, $x$ does not occur in $S$, and therefore we have
\[
B^S(t)=(B(t))^S.
\]
Thus,
$
\tilde{N}t
$
is the desired derivation of $(B(t))^S$.

If $M=\lambda_u N$, then $A=B\to C$, where $u$ is an assumption variable of $B$, and $N$ is a derivation of $C$. By the induction hypothesis, we obtain a derivation $\tilde{N}$ of $C^S$. If $u$ occurs as a free assumption in $N$, then it has been replaced by some assumption variable $\tilde{u}:B^S$. Thus,
$
\lambda_{\tilde{u}}\tilde{N}
$
is the desired derivation. If $u$ does not occur as a free assumption in $N$, we use the same term $\lambda_{\tilde{u}}\tilde{N}$, where in this case $\tilde{u}$ is a fresh assumption variable.

The proofs for the rules of $\wedge$ and the proof for elimination rule of $\to$ follow directly from the induction hypothesis. 
\end{proof}

\subsection{Connection between Classical and Intuitionistic Arithmetic}
\begin{definition}
For any formula $A$ in $\PAO$, its negative translation $A^\GG$ (also called Gödel-Gentzen translation) is recursively defined as follows:
\begin{align*}
P^\GG& := \neg\neg P \quad \text{for atomic }P\neq \FF,\\
\FF^\GG &:= \FF\\
(A\circ B )^\GG&:= A^\GG \circ B^\GG\quad \text{for }\circ\in{\rightarrow, \wedge}\\
(\forall_xA)^\GG &:= \forall_xA^\GG,\\
(A\vee B)^\GG &:= \neg(\neg A^\GG\wedge \neg B^\GG),\\
(\exists_xA)^\GG & := \neg\forall_x\neg A ^\GG,
\end{align*}
\end{definition}

We see that the Gödel--Gentzen translation translates a formula in the language of $\HAO$ or $\PAO$ into a formula of $\NAO$.
If, as in the original sources, we had defined $\NAO$ without conjunction, then we could just define $(A\wedge B)^\GG:= \neg(A^\GG\to B^\GG \to \FF)$.

\begin{theorem}
\label{thm:goedelgentzen}
Let $A$ be any statement in Peano arithmetic, then 
\begin{align*}
\PAO \vdash A \quad  \Longleftrightarrow\quad \NAO \vdash A^{\neg\neg}.
\end{align*}
\end{theorem}
\begin{proof}
See for example \cite[Section 2.3]{troelstra2000basic} or \cite[Ch.~2, Sect.~3]{troelstra1988constructivism}.
\end{proof}
The theorem above shows that classical logic (in our case \(\PAO\)) can be
embedded into the negative fragment of intuitionistic logic (in our case
\(\NAO\)).

Besides the Gödel--Gentzen translation, there are also other embeddings of
classical logic into the negative fragment of intuitionistic logic. These
translations are all similar, but they nevertheless differ in certain
respects. For our purposes the Gödel--Gentzen translation is
sufficient.
However, discussion of further possible translations, also in the context of the
\(A\)-translation, is for instance given in \cite{gaspar2012negative}.
\begin{lemma}\label{lem:GGequioverNA}
Let $A\in\form_{\NAO}$, then $A$ is equivalent to $A^\GG$ over $\NAO$, i.e. 

$$\NAO\vdash A \leftrightarrow A^\GG$$
\end{lemma}
\begin{proof}
The proof is by induction on $A$:

For $A=\FF$ there is nothing to show, as $A^\GG = A$ in this case. 

For $A=P$ for atomic $P$, we have $A=\atom{t}$ for some boolean term $t$ and $A^\GG= \neg\neg\atom t$. As $C\to \neg\neg C$ is derivable for any formula $C$, we define $$B(b) := \neg\neg\atom b\to \atom b$$ and our goal is $B(t)$. Therefore, we apply this formula to the axiom $\mathcal{C}^{b,A}$, which leads to $\forall_b(B(\true) \to B(\false) \to B(b))
$. Hence, it suffices to show $B(\true)$ and $B(\false)$. The conclusion of $B(\true)$ is $\TT$ and follows therefore from the axiom $\mathsf{Truth}$. Furthermore, $\lambda_v(v\lambda_uu)$ with assumption variables $u:\FF$ and $v:(\FF\to \FF)\to \FF$ is a derivation term for $B(\false)$.

For $A=B \circ C$ with $\circ\in\{\wedge, \to \}$, the claim follows form the induction hypothesis, i.e. $\NAO\vdash B\leftrightarrow B^\GG$ and $\NAO \vdash C\leftrightarrow C^\GG$.

For $A= \forall_x B$, we have a derivation of $B\leftrightarrow B^\GG$ in $\NAO$ by the induction hypothesis, which does not contain any free assumptions, and therefore also no free assumption about $x$. By universally generalizing both directions of the induction hypothesis, we obtain  $\forall_x(B\to B^\GG)$ and $\forall_x(B^\GG\to B)$. This leads to $\forall_x B \to \forall_x B^\GG$ and $\forall_x B^\GG \to \forall_x B$, which is exactly $\forall_x B \leftrightarrow (\forall_x B)^\GG$.

As $A\in \form_{\NAO}$ does not contain any disjunction or existential quantifier, the proof is finished.
\end{proof}
\begin{corollary}\label{coro:DerivHAToDerivNA}
Let $A\in\form_{\NAO}$ with $\HAO\vdash A$, then $\NAO\vdash A$.
\end{corollary}
\begin{proof}
A derivation of $A$ in $\HAO$ is also a derivation in $\PAO$, therefore $\PAO\vdash A$. By Theorem~\ref{thm:goedelgentzen}, we get $\NAO\vdash A^\GG$, and by Lemma \ref{lem:GGequioverNA} it follows $\NAO \vdash A$.
\end{proof}
\section{Formula Classes, Proof-Theoretic Properties, and Undecidability}
In this section, we discuss the formula classes needed for the refined
$A$-translation. These classes are defined purely by recursion on the
syntactic structure of formulas. Each of them satisfies a certain derivability
property. Thus, in the refined $A$-translation, derivability properties are
linked to the syntactic structure of formulas; this is precisely one of the
distinguishing features of the refined $A$-translation. As we shall see,
however, this connection cannot be made complete. In this sense, the link
between the syntactic structure of a formula and the derivability of a
corresponding property has its limitation.
\subsection{Case Distinction}
Since there is no strong disjunction in $\NAO$, but in some situations we
nevertheless need to do a case distinction on a formula, we say that a
formula $A$ \emph{admits case distinction} in $\XAO\in\{\NAO,\MAO,\HAO, \PAO\}$ if
$$\XAO \vdash (A\to S)\to(\neg A \to S) \to S$$
is derivable for every formula $S\in \form_{\XAO}$.
If the statement above is required only for the special case $S=\FF$, rather
than for all $S$, we obtain the weak disjunction $A\tilde{\vee}\neg A$, which is derivable for all formulas $A$.

If the strong disjunction is available, as in $\HAO$ or $\PAO$, then, by the
axiom $\vee^-$, the statement $\HAO \vdash A\vee\neg A$ is equivalent to $A$ admitting case distinction. The axiom $\mathsf{lem}$ states exactly that every formula admits case distinction in $\PAO$.

In the following, we describe the class $\CQ$ of formulas in $\NAO$ and show that all formulas in $\CQ$ admit case distinction.
\begin{definition}
With $\CP$ we denote the set of all prime formulas except $\bot$, i.e., $$\CP:=\{\atom{t}\mid t\in \term,\ t:\BB\}.$$
Furthermore, we define the set $\mathcal{Q}$ by the following rules:
\begin{align*}
\CP\subseteq \CQ,&\quad  \bot\notin \CQ,\\
A\to B,\ A\wedge B \in \CQ& \quad \text{iff}\quad A,B\in \CQ,\\
\forall_xA(x)\in \CQ& \quad \text{iff} \quad x:\BB \text{ and } A(\true),A(\false)\in \CQ.
\end{align*}
\end{definition}
Note that $\bot \notin\CQ$, and therefore formulas in $\CQ$ are formulas in $\NAO$.

\begin{lemma}
\label{lem:Qprop}
Let $\XAO\in\{\NAO,\MAO,\HAO\}$. For every $A\in\CQ$ and every
formula $S$ in $\XAO$ we have
$$
\XAO \vdash (A\to S)\to(\neg A \to S) \to S.
$$
\end{lemma}

\begin{proof}
We argue by induction on $A$ for all formulas $S$.

First, let $A=\atom{t}$ and $S$ be given and define
$$B(b):=(\atom{b} \to S)\to (\neg\atom{b}\to S) \to S.$$
Our goal is to derive $B(t)$.
By the axiom $\mathcal{C}^{b,B(b)}$ we have  
$$\forall_b(B(\true)\to B(\false)\to B(b)).$$
Thus, it is enough to derive $B(\true)$ and $B(\false)$.
These formulas are
\[
B(\true)= (\atom{\true} \to S)\to(\neg\atom{\true}\to S)\to S
\]
and
\[
B(\false)= (\atom{\false} \to S)\to(\neg\atom{\false}\to S)\to S.
\]
In both cases we have to show $S$ from the corresponding assumptions. In
the first case, $S$ follows from the first premise together with the axiom
$\mathsf{Truth}$. In the second case, $S$ follows from the second assumption,
since $\neg\atom{\false} = \atom{\false}\to \atom{\false}$ is a tautology.

Now let $A=B\to C$ and $S$ be given. We apply the induction hypothesis for $B$ to $\neg C\to S$ and the induction hypothesis for $C$ to $S$. This gives
derivations of
\begin{align}
(B\to \neg C \to S) \to (\neg B\to \neg C \to S) &\to \neg C \to S,\label{form:0}\\
(C\to S)\to (\neg C \to S) &\to S. \label{form:1}
\end{align}
Our goal is $$((B\to C) \to S) \to (\neg(B\to C)\to S) \to S.$$
Therefore, we assume 
$$
(B\to C)\to S
\qquad\text{and}\qquad
\neg(B\to C)\to S
$$
and our goal is $S$. By \eqref{form:1}, it is enough to derive $C\to S$ and
$\neg C\to S$. The formula $C\to S$ follows directly from $(B\to C)\to S$, since $C$ implies $B\to C$. It remains to derive $\neg C\to S$. By \eqref{form:0}, it is enough
to derive
$$
B\to \neg C\to S
\qquad\text{and}\qquad
\neg B\to \neg C\to S .
$$
The first formula follows from $\neg(B\to C)\to S$. For the second formula, assume $\neg B$ and $\neg C$.
We have to derive $S$. Since $(B\to C)\to S$ is already available, it is
enough to derive $B\to C$. Thus assume $B$. Together with $\neg B$, this gives
$C$ by Lemma~\ref{lem:efq}.

Next, let $A=B\wedge C$ and $S$ be given. We apply the induction hypothesis for $B$ to $C\to S$ and the induction hypothesis for $C$ to $S$. This gives
derivations of
\begin{align}
(B\to C\to S)\to (\neg B \to C \to S) &\to C \to S\label{form:2}\\
(C\to S)\to (\neg C \to S) &\to S\label{form:3}
\end{align}
Our goal is
$$
(B\wedge C \to S) \to (\neg(B\wedge C) \to S) \to S .
$$
Assume therefore
$$
B\wedge C \to S
\qquad\text{and}\qquad
\neg(B\wedge C)\to S .
$$ and our goal is $S$.
By \eqref{form:3} it suffices to show  $C\to S$ and $\neg C \to S$. The second formula follows directly by $\neg(B\wedge C) \to S$ ,
because $\neg C$ implies $\neg(B\wedge C)$.
It remains to derive $C\to S$.
Using \eqref{form:2}, it is enough to derive
$$
B\to C\to S
\qquad\text{and}\qquad
\neg B\to C\to S .
$$
The first formula follows from $B\wedge C\to S$. The second formula follows
from $\neg(B\wedge C)\to S$, since $\neg B$ implies $\neg(B\wedge C)$.

Finally, let $A=\forall_b B(b)$, where $b$ has type $\BB$. By the axiom
$\mathcal{C}^{b,B(b)}$, the formula $\forall_b B(b)$ is equivalent to
$B(\true)\wedge B(\false)$. Hence this case is treated in the same way as the
case for conjunction.
\end{proof}

With the class $\CQ$ and the lemma above, we have identified a class of
formulas whose elements satisfy a certain derivability property. In what follows, we shall only need this property.
Originally, $\CQ$ was intended to be the class of all quantifier-free formulas
not containing $\bot$. By allowing quantification over Boolean objects, we have
slightly extended this class here. One could also extend it further, for example by adding any derivable formula.

As in the case of the originally formulated open problem, the question arises
whether the class of all formulas satisfying the property stated in the lemma
above can be characterized in a useful way. As the following proposition shows,
however, we already reach a limitation at this point:

\begin{proposition}\label{prop:casesUndecidable}
Assuming that $\PAO$ is  consistent, there is no recursive algorithm which decides whether a given closed formula $A\in\operatorname{Form}_{\NAO}$ has the property $$\MAO \vdash (A\to \bot)\to(\neg A \to \bot) \to \bot.$$
\end{proposition}
\begin{proof}
Applying Lemma~\ref{lem:substbot} to $A\vee\neg A\in\form$ and the rules $\vee^+_{0/1}$, $\MAO \vdash (A\to \bot)\to(\neg A \to \bot) \to \bot$ leads to $\HAO\vdash A\vee\neg A$. In the other direction if $\HAO\vdash A\vee\neg A$, then by the \emph{disjunction property} for closed formulas (see for example \cite[Corollary 5.24]{kohlenbach2008applied} or \cite[Section 1.11.2]{troelstra1973metamathematical}) we get $\HAO\vdash A$ or $\HAO \vdash \neg A$ and by Corollary~\ref{coro:DerivHAToDerivNA} either $\NAO\vdash A$ or $\NAO\vdash \neg A$. In both cases we get $\MAO \vdash (A\to \bot)\to(\neg A \to \bot) \to \bot$. Therefore, a characterization of all closed $A$ with $\MAO \vdash (A\to \bot)\to(\neg A \to \bot) \to \bot$ is equivalent to a characterization of all closed $A$ with $\HAO\vdash A\vee\neg A$. We show that this is not possible.

For this, let a Turing machine $T$ (with arbitrary but fixed input) be given. By \cite[Chapter 42]{smith2007introduction}, there is some $\tilde{A}\in\form$ with $$\PAO\vdash \tilde{A}\quad \Leftrightarrow\quad T \text{ terminates}.$$
By Theorem \ref{thm:goedelgentzen} we can even assume that there is some $A\in\form_{\NAO}$ with
$$\NAO\vdash A\quad \Leftrightarrow\quad T \text{ terminates}.$$
If we had a recursive algorithm which decides $\HAO\vdash A\vee\neg A$ for any given $A\in \form_{\NAO}$, we 
could decide, 
$$ \HAO\vdash A\vee \neg A \qquad \text{or}\qquad \HAO \not\vdash  A\vee \neg A.$$
If the left-hand side holds we get either $\HAO\vdash A$ or $\HAO\vdash\neg A$ by the disjunction property. By our assumption $\PAO$ and thereby $\HAO$ is consistent, and therefore we get that $T$ terminates or not in this case. If the right-hand case holds, then $\HAO\not\vdash A$ follows, as $\HAO\vdash A$ implies $\HAO\vdash A\vee \neg A$ by $\vee^+_0$. Hence an algorithm to decide $\HAO\vdash A\vee\neg A$  would lead to an algorithm which decides whether any given Turing machine terminates or not, which is not possible \cite[Theorem 43.4]{smith2007introduction}.
\end{proof}

\subsection{Definite, Goal, Relevant and Irrelevant Formulas}
We now give the definitions of definite ($\CD$), goal ($\CG$), relevant ($\CR$), and irrelevant formulas ($\CI$). To show that the definition of these formulas is indeed well-founded and terminates without entering an infinite loop, we give it by recursion on the structure of formulas. Thus, it is somewhat more complex than the original definition in  \cite{berger2002refined,schwichtenberg2012proofs,seisenberger2003constructive,trifonov2012analysis}.
\begin{definition}
\label{defi:DGRI}
Let $\CQ_\FF:= \{ A\in\operatorname{Form}_{\MAO} \mid A^\FF \in \CQ \}$, we define the sets $\CD, \CG, \CR, \CI \subseteq \operatorname{Form}_{\MAO}$ by induction on the formula structure as follows:
\begin{align*}
\bot\in \CD\cap \CR\cap  \CG,\qquad \bot \notin \CI
\end{align*}
\begin{align*}
\CP \subseteq \CD\cap  \CI\cap  \CG,\qquad \CP\cap \CR =\{\TT \}
\end{align*}
\begin{align*}
A\to B \in 
\begin{cases} 
\CD &\text{ iff }\quad A\in \CI \text{ and } B\in \CD \quad \text{or}\quad A\in\CG \text{ and } B\in\CR\\
\CG &\text{ iff }\quad A\in \CR\cup\CD\cap {\CQ_\FF} \text{ and } B\in \CG\quad \text{or}\quad A\in\CD \text{ and } B\in\CI \\
\CR &\text{ iff }\quad A\in \CG \text{ and } B\in \CR \\
\CI &\text{ iff }\quad A\in \CD \text{ and } B\in \CI \\
\end{cases}
\end{align*}
\begin{align*}
A\wedge B \in \mathcal{S}\quad \text{iff}\quad A,B\in \mathcal{S}\quad \text{ with } \mathcal{S}\in\{\CD, \CG, \CR, \CI \},
\end{align*}
\begin{align*}
\forall_xA(x) \in \begin{cases} 
\CD&\text{ iff }\quad A\in \CD\cup \CR\\
\CG&\text{ iff }\quad A\in \CI\quad \text{or}\quad  x:\BB \ \text{and}\ A(\true),A(\false)\in \CG\\
\CR&\text{ iff }\quad A\in \CR\\
\CI&\text{ iff }\quad A\in \CI\\
\end{cases}
\end{align*}
Note, that $\CR\subseteq \CD$ and $\CI\subseteq \CG$.
\end{definition}
\begin{lemma}\label{lem:set_props}
The following formulas are derivable in $\MAO$:
\begin{align}\label{form:prop1}
D^\FF &\to D,\\\label{form:prop2}
G &\to (G^\FF \to \bot) \to \bot,\\\label{form:prop3}
(\neg R^\FF\to \bot) &\to R,\\\label{form:prop4}
I&\to I^\FF,
\end{align}
where $D\in\CD$, $G\in \CG$, $R\in \CR$ and $I\in \CI$.
\end{lemma}
A proof of this statement was already given in \cite{schwichtenberg2012proofs}.
However, we have slightly extended the theory, so that additional cases
arise. Moreover, the structure of the proof is very useful for our
considerations on extending the classes of formulas. We therefore include
the full proof here as well.
\begin{proof}
The proof is by simultaneous induction on the formula structure:\bigskip

\underline{Case $\bot$:}
We have $\FF\to \bot$ by $\bot^+$, which proves \eqref{form:prop1} in this case.
Furthermore, $\bot\to(\FF\to\bot) \to \bot$ and $((\FF \to\FF)\to\ \bot) \to \bot$ are trivial, hence \eqref{form:prop2} and \eqref{form:prop3}. Note that for \eqref{form:prop4} there is nothing to show as $\bot\notin \CI$.
\medskip

\underline{Case $\atom t$:}
Note that $\atom{t}^\FF = \atom{t}$. Therefore \eqref{form:prop1}, \eqref{form:prop2} and \eqref{form:prop4} are obvious. By the definition of $\CR$, we only have to show \eqref{form:prop3} for $R=\TT$, which follows directly by the axiom~$\mathsf{Truth}$.
\bigskip

\underline{Case $A\to B$:}
\medskip

\underline{Subcase 1.1}: Let $A\to B\in\CD$ with $A\in \CI$ and $B\in \CD$. By induction hypothesis we have derivations of $A\to A^\FF$ and $B^\FF\to B$ and our goal is $$(A^\FF \to B^\FF)\to A \to B.$$ Hence, assume $A^\FF \to B^\FF$ and $A$, to prove $B$. From $A$ and $A\to A^\FF$ we get $A^\FF$ and with $A^\FF \to B^\FF$ we get $B^\FF$. Therefore, $B^\FF\to B$ gives the goal $B$.

\underline{Subcase 1.2}: Let $A\to B\in\CD$ with $A\in \CG$ and $B\in \CR$.

By the induction hypothesis, we have derivations of $A\to (A^\FF \to \bot) \to \bot$ and $(\neg B^\FF\to \bot)\to B$. We have to show $(A^\FF \to B^\FF)\to A \to B$. Thus, assume $A^\FF \to B^\FF$ and $A$. Our goal is to prove $B$. Using the derivation of $(\neg B^\FF\to \bot)\to B$, it is enough to prove $\neg B^\FF\to \bot$. So let us assume $\neg B^\FF$. We have to derive $\bot$.

Since $A$ and $A\to (A^\FF \to \bot) \to \bot$ are already given, it remains to prove $A^\FF \to \bot$. For this, assume $A^\FF$. Then, from $A^\FF \to B^\FF$, we obtain $B^\FF$. Together with the assumption $\neg B^\FF$, this gives $\FF$. Hence, by the axiom $\bot^+$, we obtain $\bot$, as required.
\medskip

\underline{Subcase 2.1}: Let $A\to B\in\CG$ with $A\in \CR$ and $B\in \CG$.
By the induction hypothesis, we have derivations of $(\neg A^\FF \to \bot) \to A$ and $B\to (B^\FF\to \bot)\to \bot$. We have to show
\[
(A\to B)\to ((A^\FF\to B^\FF)\to \bot)\to\bot.
\]
Thus, assume $A\to B$ and $(A^\FF\to B^\FF)\to \bot$. Our goal is $\bot$. By the derivation of $B\to (B^\FF\to \bot)\to \bot$, it is enough to prove both $B$ and $B^\FF\to \bot$. The second formula follows directly from the assumption $(A^\FF\to B^\FF)\to \bot$, since $B^\FF$ yields $A^\FF\to B^\FF$.

It remains to prove $B$. Using $A\to B$, it is enough to prove $A$. By the derivation of $(\neg A^\FF \to \bot) \to A$, it suffices to show $\neg A^\FF \to \bot$. So assume $\neg A^\FF$, i.e., $A^\FF\to \FF$. By Lemma~\ref{lem:efq}, this gives $A^\FF\to B^\FF$. Together with the assumption $(A^\FF\to B^\FF)\to \bot$, this yields $\bot$, as required.
\smallskip

\underline{Subcase 2.2}: Let $A\to B\in\CG$ with $A\in \CD \cap \CQ_\FF$ and $B\in \CG$. By the induction hypothesis, we have derivations of $A^\FF\to A$ and $B\to (B^\FF\to \bot)\to \bot$.  Again, we assume $A\to B$ and $(A^\FF\to B^\FF)\to \bot$, and our goal is $\bot$. By Lemma~\ref{lem:Qprop}, we have a derivation of $(A^\FF\to \bot)\to(\neg A^\FF\to \bot) \to \bot$. Hence, it is enough to prove both $A^\FF\to \bot$ and $\neg A^\FF\to \bot$. The second formula follows as at the end of the previous case: by Lemma~\ref{lem:efq}, the assumption $\neg A^\FF$ gives $A^\FF\to B^\FF$, and together with $(A^\FF\to B^\FF)\to \bot$ this yields $\bot$.

It remains to prove $A^\FF\to \bot$. So assume $A^\FF$. We have to derive $\bot$. By the derivation of $B\to (B^\FF\to \bot)\to \bot$, it suffices to prove $B$ and $B^\FF\to \bot$. The formula $B^\FF\to \bot$ follows directly from $(A^\FF\to B^\FF)\to \bot$. To prove $B$, we use $A^\FF$ and the derivation of $A^\FF\to A$ to obtain $A$, and then $B$ follows from the assumption $A\to B$.
\smallskip

\underline{Subcase 2.3}:
Let $A\to B\in\CG$ with $A\in \CD$ and $B\in \CI$. By the induction hypothesis, we have derivations of $A^\FF \to A$ and $B\to B^\FF$. We again assume $A\to B$ and $(A^\FF \to B^\FF)\to \bot$. Our goal is $\bot$. By the assumption $(A^\FF \to B^\FF)\to \bot$, it is enough to prove $A^\FF \to B^\FF$. Thus, assume $A^\FF$. Then we obtain $A$ from the derivation of $A^\FF \to A$. Hence $A\to B$ gives $B$, and from $B\to B^\FF$ we obtain $B^\FF$. This proves $A^\FF \to B^\FF$ as required.
\medskip

\underline{Subcase 3}: Let $A\to B\in\CR$ with $A\in \CG$ and $B\in \CR$. By the induction hypothesis, we have derivations of $A\to (A^\FF\to \bot) \to \bot$ and $(\neg B^\FF\to \bot) \to B$. We have to show
\[
(\neg(A^\FF\to B^\FF)\to \bot) \to A\to B.
\]
Thus, assume $\neg(A^\FF\to B^\FF)\to \bot$ and $A$. Our goal is $B$. By the derivation of $(\neg B^\FF\to \bot) \to B$, it is enough to prove $\neg B^\FF\to \bot$. So assume $\neg B^\FF$. We have to derive $\bot$. Using the derivation of $A\to (A^\FF\to \bot) \to \bot$ together with the assumption $A$, it remains to prove $A^\FF\to \bot$. Hence, assume $A^\FF$. Our goal is again $\bot$. By the assumption $\neg(A^\FF\to B^\FF)\to \bot$, it suffices to prove $\neg(A^\FF\to B^\FF)$. But this follows directly from the assumptions $A^\FF$ and $\neg B^\FF$: indeed, any proof of $A^\FF\to B^\FF$ would yield $B^\FF$, contradicting $\neg B^\FF$. Hence we obtain $\bot$, as required.
\medskip 

\underline{Subcase 4}: Let $A\to B\in \CI$ with $A\in \CD$ and $B\in \CI$. By induction hypothesis we have derivations of $A^\FF\to A$ and $B\to B^\FF$ and our goal is
$$ (A \to B)\to A^\FF \to B^\FF.$$
The proof is similar to the proof of Case~1.1, with the roles of $A$,$B$ and $A^\FF$, $B^\FF$ interchanged.
\bigskip

\underline{Case $A\wedge B$:}
In this case, the properties \eqref{form:prop1}--\eqref{form:prop4} have to be transferred from $A$ and $B$ to $A\wedge B$. This is straightforward for all properties except \eqref{form:prop2} concerning $\CG$. In order not to make the proof unnecessarily long, we only consider this case.

Our goal is to derive
$$
A\wedge B \to (A^\FF\wedge B^\FF \to \bot) \to \bot.
$$
Thus, assume $A\wedge B$ and $A^\FF\wedge B^\FF \to \bot$.
By the induction hypothesis we have $A \to (A^\FF \to \bot) \to \bot$ and  $B \to (B^\FF \to \bot) \to \bot$.
From the assumption $A\wedge B$ we get $(A^\FF \to \bot) \to \bot$ and  $(B^\FF \to \bot) \to \bot$.
By $(A^\FF\to \bot)\to \bot$, it suffices to prove $A^\FF\to \bot$. So assume $A^\FF$. We now have to show $\bot$. By $(B^\FF\to \bot)\to \bot$, it is enough to prove $B^\FF\to \bot$. Therefore we assume $B^\FF$ and the goal is $\bot$. Together with the assumption $A^\FF$ we get $A^\FF\wedge B^\FF$. Using the assumption $A^\FF\wedge B^\FF \to \bot$ we obtain $\bot$, as required.\bigskip 

\underline{Case $\forall_xA$}:
In this case, note that the respective derivation obtained from the induction hypothesis applied to $A$ has no free assumptions and, in particular, does not make any assumption about $x$.

\underline{Subcase 1.1}: Let $\forall_x A\in \CD$ with $A\in \CD$. We have to show $$\forall_x A^\FF \to \forall_x A.$$ Thus, assume $\forall_x A^\FF$. We fix an arbitrary $x$ and show $A$. By instantiating the assumption $\forall_x A^\FF$ at $x$, we obtain $A^\FF$. By the induction hypothesis, we have a derivation of $A^\FF\to A$, and hence we obtain $A$. The only open assumption is $\forall_x A^\FF$, in which $x$ does not occur freely. Therefore, we may generalize over $x$ to get $\forall_x A$.
\smallskip

\underline{Subcase 1.2}: Let $\forall_x A\in \CD$ with $A\in \CR$. We again assume $\forall_x A^\FF$ and fix an arbitrary $x$ in order to show $A$. By the induction hypothesis applied to $A$, we have a derivation of $(\neg A^\FF \to \bot) \to A$. Thus, it is enough to prove $\neg A^\FF \to \bot$. So assume $\neg A^\FF$. We have to derive $\bot$. By instantiating the assumption $\forall_x A^\FF$ at $x$, we obtain $A^\FF$. Together with $\neg A^\FF$, this yields $\FF$, and by $\bot^+$ we obtain $\bot$, as required. Hence $\neg A^\FF \to \bot$ is proven and we may conclude $A$. Since $x$ is not free in the only open assumption $\forall_xA^\FF$, we obtain $\forall_x A$.
\medskip

\underline{Subcase 2.1}: Let $\forall_x A\in \CG$ with $A\in \CI$. By the induction hypothesis, we have a derivation of $\forall_x(A \to A^\FF)$. We have to show
\[
\forall_x A \to (\forall_x A^\FF\to \bot) \to \bot.
\]
Thus, assume $\forall_x A$ and $\forall_x A^\FF\to \bot$. Our goal is $\bot$. We use the assumption $\forall_x A^\FF\to \bot$; hence it remains to prove $\forall_x A^\FF$. So let $x$ be fixed but arbitrary. From the induction hypothesis we obtain $A\to A^\FF$. Moreover, from the assumption $\forall_x A$, instantiated at $x$, we obtain $A$. Therefore we get $A^\FF$. Since $x$ was arbitrary, this proves $\forall_x A^\FF$, and hence $\bot$, as required.
\smallskip

\underline{Subcase 2.2}: Let $\forall_x A(x)\in \CG$ with $x:\BB$ and $A(\true),A(\false)\in \CG$.
By the axiom
$$
\mathcal{C}^{x,A}:\forall_x\bigl(A(\true)\to A(\false)\to A(x)\bigr),
$$
the formula $\forall_x A$ is equivalent to
$
A(\true)\wedge A(\false).
$
Thus, this case is analogous to the case $A\wedge B$.

\medskip

\underline{Subcase 3}: Let $\forall_x A\in \CR$ with $A\in \CR$. We have to show
\[
(\neg\forall_x A^\FF \to \bot)\to \forall_x A.
\]
Thus, assume $\neg\forall_x A^\FF \to \bot$. We fix an arbitrary $x$ and prove $A$. By the induction hypothesis, we have a derivation of $(\neg A^\FF \to \bot)\to A$. Hence it is enough to prove $\neg A^\FF \to \bot$. So assume $\neg A^\FF$. We have to derive $\bot$. By the assumption $\neg\forall_x A^\FF \to \bot$, it suffices to prove $\neg\forall_x A^\FF$. Therefore, assume $\forall_x A^\FF$. Then, by instantiating this assumption at $x$, we obtain $A^\FF$. Together with $\neg A^\FF$, this yields $\FF$. Hence we have proved $\neg\forall_x A^\FF$, and therefore obtain $\bot$. This proves $A$, and since $x$ is not free in the only free assumption $\neg\forall_x A^\FF \to \bot$, we conclude $\forall_x A$.
\medskip

\underline{Subcase 4}: Let $\forall_x A\in \CI$ with $A\in \CI$. The proof in this case is similar to the proof of Case~1.1, with the roles of $A$ and $A^\FF$ interchanged.

\end{proof}

\begin{remark}
\label{Remark:AddST}
Let $A$ be any prime formula which is not $\bot$ and not $\FF$. We define 
\begin{align*}
S&:=\forall_x(\neg\neg A\to A)\\
T&:= (\forall_xA\to \bot) \to \bot.
\end{align*}
In \cite[Section 7.3.1]{schwichtenberg2012proofs} Schwichtenberg and Wainer explained that although $(S\to T)^\FF \to (S\to T)$ is derivable, $S\to T$  does not lie in $\CD$. Thus, the set $\CD$ can still be extended, and as an open problem it was asked
``to find a useful characterization of the class of formulas such that
$D^F \to D$ is provable intuitionistically''. The following proposition gives
a ``solution'' to this open problem by showing that it is not solvable:
\end{remark}
%
%
\begin{proposition}\label{prop:SetsUndec}
Assuming that $\PAO$ is consistent, for any of the properties (\ref{form:prop1}) - (\ref{form:prop4}) in Lemma \ref{lem:set_props} there is no recursive algorithm to decide whether a given formula $A$ has this property or not.
\end{proposition}
\begin{proof}
We begin with the property $D^\FF\to D$.
Let $A$ be an arbitrary formula in $\NAO$, and define $D:=\bot \to A$. Then $D^\FF = \FF \to A$ is derivable by Lemma~\ref{lem:efq}. Hence, there is a proof of $D^\FF \to D$ in $\MAO$ if and only if $\bot \to A$ is derivable in $\MAO$. By Lemma~\ref{lem:substbot}, a derivation of $\bot \to A$ yields a derivation of $\TT \to A$ in $\NAO$, which is equivalent to a derivation of $A$ itself. In the other direction, any derivation of $A$ in $\NAO$ immediately gives a derivation of $(\FF \to A)\to \bot \to A$ in $\MAO$. 
Thus, a characterization of all formulas $D$ satisfying $D^\FF\to D$ would yield a characterization of all derivable formulas in $\NAO$. By Theorem~\ref{thm:goedelgentzen}, this would in turn yield a characterization of all derivable formulas in $\PAO$. Such a recursive characterization is impossible; see, for instance, \cite[Theorem~40.1]{smith2007introduction}.

For the property $G\to (G^\FF \to \bot)\to \bot$, let any formula $A$ in $\NAO$ be given and define $G:=A\to\bot$. Then, $G\to (G^\FF \to \bot)\to \bot = (A\to \bot) \to (\neg A \to \bot) \to \bot$ and the claim follows directly by Proposition~\ref{prop:casesUndecidable}.

For the property $(\neg R^\FF \to \bot) \to R$, let $A$ be a formula in $\NAO$ and define $R:=A$. Then, by Lemma  \ref{lem:substbot}, $\MAO \vdash (\neg R^\FF \to \bot) \to R$ leads to  $\NAO \vdash (\neg A \to \TT) \to A$, which is equivalent to $\NAO \vdash A$. In the other direction: a derivation of $A$ in $\NAO$ leads directly to a derivation of $(\neg R^\FF \to \bot) \to R$ in $\MAO$. Analogously to the first property we get, that such a recursive characterization is not possible.

For the property $I\to I^\FF$ we take an arbitrary formula $A$ and define $I:= A\to \bot$. Then, similar to the first property, a derivation of $I\to I^\FF$ in $\MAO$ is equivalent to a derivation of $\neg A$ in $\NAO$ by substituting $\TT$ for $\bot$. As any formula $A$ is equivalent $\neg\neg A$ in $\PAO$, $\PAO \vdash A$ is equivalent to $\NAO \vdash \neg \neg A^\GG$ 
by Theorem  \ref{thm:goedelgentzen}. Hence, a characterization of all $A$ with $\NAO \vdash \neg A$  could be applied to formulas of the form $\neg A^\GG$,
and would therefore decide whether $\NAO\vdash \neg\neg A^\GG$, which would lead to a characterization of all $A$ with $\PAO \vdash A$, which is recursively impossible.
\end{proof}
\section{\texorpdfstring{Refined $A$-Translation}{Refined A-Translation}}
\subsection{\texorpdfstring{The Theorem on the Refined $A$-Translation}{The Theorem on the Refined A-Translation}}
In this section, we now state the theorem of refined $A$-translation. The idea of the refined $A$-translation is to transform a weak existence statement of the form
\[
\forall_x(A\to \bot)\to \bot
\]
into a strong existence statement
$
\exists_x A
$
by replacing $\bot$ with $\exists_x A$. Then
$
\forall_x(A\to \exists_x A)
$
follows directly from $\exists^+$, and what remains is $\exists_x A$.

Note, however, that this is not yet an exact substitution, and is therefore not entirely correct as stated. The reason is that both $A$ and certain axioms or assumptions occurring in the proof of the weak existence statement may themselves contain $\bot$. 

Nevertheless, based on this idea and on the formal considerations from the previous two sections, we obtain the following theorem:
\begin{theorem}\label{Thm:ATrans}
Let $D$ be a formula such that
$
\MAO \vdash D^\FF \to D,
$
and let $G$ be a formula such that
$
\MAO \vdash \forall_x\bigl(G \to (G^\FF \to \bot) \to \bot\bigr).
$\\
Assume moreover that
$
\MAO \vdash D \to \forall_x(G\to\bot) \to \bot.
$
Then
\[
\HAO \vdash D^\FF \to \exists_x G^\FF.
\]
\end{theorem}
\begin{proof}
First we show 
\begin{align}\label{form:MAGoal}
\MAO \vdash D^\FF \to \forall_x(G^\FF\to \bot) \to \bot.
\end{align}
To this end, we work in $\MAO$ and assume $D^\FF$ and
$\forall_x(G^\FF\to \bot)$.

Our goal is to derive $\bot$. From $\MAO \vdash D^\FF\to D$ and the assumption $D^\FF$, we obtain $D$. Hence, using
$
\MAO \vdash D \to \forall_x(G\to\bot) \to \bot,
$
it remains to prove $\forall_x(G\to\bot)$.\\
Thus, let $x$ be arbitrary and assume $G$, and we have to derive $\bot$.\\
By
$
\MAO \vdash \forall_x\bigl(G\to (G^\FF\to\bot)\to\bot\bigr),
$
it is enough to prove $G$ and $G^\FF\to\bot$. The formula $G$ is one of our assumptions. Moreover, $G^\FF\to\bot$ follows directly from the assumption
$
\forall_x(G^\FF\to\bot).
$
Hence, \eqref{form:MAGoal} is proved.

We now apply Lemma~\ref{lem:substbot} to the formula $\exists_x G^\FF\in\form_{\HAO}$ and to \eqref{form:MAGoal}. This yields
\[
\HAO \vdash
D^\FF \to
\forall_x\bigl(G^\FF\to \exists_x G^\FF\bigr)
\to
\exists_x G^\FF.
\]
As
$
\forall_x\bigl(G^\FF\to \exists_x G^\FF\bigr)
$
follows directly from the axiom $\exists^+$, we have
$
\HAO \vdash D^\FF \to \exists_x G^\FF,
$
as required.
\end{proof}
We see that the properties of $D$ and $G$ are those from Lemma~\ref{lem:set_props}. Thus, we directly obtain the following corollary, which is regarded in the sources~\cite{berger2002refined, schwichtenberg2012proofs,seisenberger2003constructive, trifonov2012analysis} as the refined $A$-translation:
\begin{corollary}\label{cor:ATrans}
Let $D\in \CD$ and $G\in \CG$  such that $\MAO \vdash D \to \forall_x (G\to\bot) \to \bot$. Then $$\HAO\vdash D^\FF \to \exists_x G^\FF.$$
\end{corollary}
\begin{proof}
This follows directly from Theorem \ref{Thm:ATrans} and Lemma \ref{lem:set_props}. Note that the property for $G\in\CG$ can be universally quantified, since the derivation in Lemma~\ref{lem:set_props} has no free assumptions.
\end{proof}
\begin{remark}
Note that in the original sources (e.g.~\cite{berger2002refined, schwichtenberg2012proofs, trifonov2012analysis}), the more general form
\[
\MAO \vdash D_0 \to \dots \to D_{n-1}
\to
\forall_x(G_0\to \dots \to G_m \to \bot)
\to \bot
\]
for arbitrary $n,m\geq 0$ was considered. In our setting, however, we have conjunction available in $\NAO$. Therefore, this form follows directly from the corollary above by taking
\[
D := \begin{cases}
D_0 \wedge \dots \wedge D_{n-1} &\text{if } n>0\\
\TT &\text{else}
\end{cases}
\]
and
$
G := G_0 \wedge \dots \wedge G_m.
$
\end{remark}
\subsection{Applications and an Assessment of the Limits of the Recursive Characterization}
In this section, we discuss applications of the refined A-translation, with particular emphasis on the implications of the limitations established in the previous section.

The following table lists several applications of the refined $A$-translation together with the corresponding references. No claim of completeness is made. However, the table shows that the refined A-translation is applicable to typical examples in constructive mathematics.  Most of these applications were implemented in the proof assistant Minlog, thereby enabling automatic program extraction. 

\begin{table}[ht]
\centering
\begin{tabular}{p{0.5\textwidth} p{0.4\textwidth}}
\hline
\textbf{Application of the refined $A$-translation} & \textbf{Reference} \\
\hline
Existence of the greatest common divisor & \cite{berger1996greatest}\\
Higman's lemma & \cite[Chapter 6]{seisenberger2003constructive} \\
Existence of the Fibonacci graph &\cite[Sec.~6.1]{berger2002refined}, \cite[Sec.~7.3.4]{schwichtenberg2012proofs} \\
Well-foundedness of the natural numbers &\cite[Sec.~6.2]{berger2002refined}, \cite[Sec.~7.3.5]{schwichtenberg2012proofs} \\
The $hsh$-theorem & \cite[Sec.~7.3.6]{schwichtenberg2012proofs}\\
Stolzenberg's example, infinite tape example  & \cite[Sec.~5.1, 6.1]{ratiu2011refinement}, \cite{seisenberger2008programs}, \cite[Sec.~3.1]{trifonov2012analysis} \\
Existence of the integer root & \cite[Sec.~3.2]{trifonov2012analysis}\\
The pigeonhole principle &\cite[Chapter~5]{ratiu2011refinement}, \cite[Sec.~3.3]{trifonov2012analysis}  \\
Erd\H{o}s--Szekeres theorem & \cite[Sec.~6.2]{ratiu2011refinement}\\
A special case of Dickson's lemma & \cite{berger2001warshall}, \cite[Sec.~7]{ratiu2011refinement} \\
\hline\noalign{\vskip 5pt}
\end{tabular}
\caption{Applications of the refined $A$-translation and related references}
\label{tab:a-translation-applications}
\end{table}

At first sight, one might say that Theorem~\ref{Thm:ATrans} is the better version, since it is more general, and that Corollary~\ref{cor:ATrans} is only a special case of it. However, Corollary~\ref{cor:ATrans} has the advantage that $D\in\CD$ and $G\in\CG$ are decidable conditions. Thus, in this case, one directly obtains computational content by a recursive procedure, which makes the refined $A$-translation particularly well suited for computer implementation. By contrast, the computational content of Theorem~\ref{Thm:ATrans} is contained in the proofs of the two required properties. As we can see from the table above, the version in Corollary \ref{cor:ATrans} is already quite sufficient. 
The counterexample mentioned in Remark \ref{Remark:AddST} was also taken up by Ratiu in \cite[pp.~50-51]{ratiu2011refinement}. There she writes that, so far, no relevant case study of this form is known. It therefore seems that the existing form of the formula classes has in fact been sufficient so far.

However, this does not mean that we should stop extending these classes.
The formulas $S$ and $T$ in Remark~\ref{Remark:AddST} were chosen such that
$S\to T$ is not in $\CD$, although there is a proof of
\[
(S^\FF\to T^\FF)\to S\to T.
\]
Whether a formula has the form $S\to T$ can be checked recursively. Hence, formulas of this form can be added to the definition of $\CD$.
If one finds further, or larger, recursively definable classes of formulas satisfying one of the properties from Lemma~\ref{lem:set_props}, then these classes could be added to the corresponding class. In this way, one can build a repository of recursively checkable formula classes which, for a given formula, checks which of the sets it belongs to and, in the successful case, provides a proof. 

This suggests the following new open problem: find a meaningful application of Theorem \ref{Thm:ATrans} which is not already covered by Corollary \ref{cor:ATrans}, and which would therefore motivate a genuine extension of the classes $\CD$, $\CG$, $\CR$, and $\CI$.

\section{Remarks on the Rust Prover}
In the final section, we briefly discuss the insights we gained from implementing the theory $\MAO$ in Rust. The implementation was originally developed in order to test the individual classes of formulas. Interestingly, there does not seem to be a major proof assistant written in Rust so far. For this reason, we briefly describe our experiences with the implementation. With this discussion, we may be able to inspire further development of proof assistants in Rust. The precise implementation can be found in the GitHub repository \cite{wiesnet2026checker}. We only assume the basics of Rust, but no deeper understanding of the language. Readers who would like to learn more about Rust and study the language in greater detail are referred to the official Rust book \cite{klabnik2025rust}.

\subsection{Data Types}
The strong type system of Rust was very helpful when writing the proof assistant and is therefore probably one of the greatest advantages that Rust, as a programming language, offers for implementing a proof assistant.

\subsubsection{Data Structures: \texttt{enum} and \texttt{struct}}
The individual syntactic categories, such as types, terms, formulas, and proofs, can be defined particularly conveniently using \texttt{enum} and \texttt{struct}. In particular, the pattern matching associated with \texttt{enum} is especially useful for recursively defined objects. Pattern matching plays an important role in almost every proof assistant. Agda, for example, is particularly well known for this.
The question of the extent to which Rust should be regarded as an object-oriented language is controversial. For the development of a proof assistant, however, Rust provides several features that are useful from an object-oriented perspective.

In addition to \texttt{enum}, \texttt{struct} was particularly useful for defining terms and proofs, especially in combination with associated methods and encapsulation. This makes it possible to hide the internal structure of terms and proofs, so that these objects can only be constructed through functions that check whether the corresponding construction or proof rule has been applied correctly.

Polymorphism, in particular through traits, was not substantially used in the current implementation. Nevertheless, it seems plausible that this feature could be useful in future extensions. For example, types and formulas, as well as terms and proofs, share certain structural similarities, and one could try to capture common operations by means of suitable traits. In this way, Rust's trait system might help to design a more general proof assistant framework that can be instantiated for different formal systems or metatheories. A challenge that did not arise in the case of $\MAO$ is the
treatment of dependent types. The question of how to define them
efficiently in Rust remains open and is very interesting. It is possible
that traits could be useful for this purpose. This, however, is beyond the scope of the present implementation and remains a topic for future research.

\subsubsection{Partial Constructions: \texttt{Option} and \texttt{Result}}
Rust's types \texttt{Option} and \texttt{Result} were also useful for implementing partial constructions. 
\texttt{Result<T,E>} makes potential errors explicit. Here, \texttt{T} is any data type and \texttt{E} is an Error Datatype. It is particularly suitable for a proof assistant, where term-construction rules or proof rules are only applicable under certain conditions. For example, if two terms $s,t$ are given, then the term application $st$
only yields a meaningful new term if $s$ has a function type
$\tau \to \rho$ (for some types $\tau$ and $\rho$) and $t$ has the type $\tau$. In this case, the
resulting term can be returned. In all other cases, one obtains a
\texttt{TypeError}.

For the definition of the sets $\CD,\CG,\CR$ and $\CI$ in the Rust
implementation, \texttt{Option<T>} was particularly useful.
The type \texttt{Option<T>} either returns an object of data type \texttt{T}
or \texttt{None}. In this way, the sets could be defined as functions
which take a formula as input and either return \texttt{None}, if the
formula does not belong to the set, or directly return a proof of the
corresponding properties \eqref{form:prop1}--\eqref{form:prop4}. Thus, Definition~\ref{defi:DGRI} and Lemma~\ref{lem:set_props} were
implemented in one function.

\subsection{Ownership, Borrowing, and Smart Pointers for Provers}

Rust's ownership system and borrow checker are, in principle, very useful
features. They provide strong guarantees about memory safety and prevent
many common programming errors. However, in the implementation of a proof
assistant they can also make the code more verbose.

This is especially visible for recursively defined objects such as terms,
formulas, and proofs. Since such objects do not have a statically known
finite size, their definition in Rust requires the use of smart pointers,
in our case \texttt{Box}. As a consequence, the code becomes more explicit
and sometimes more cumbersome than in programming languages with automatic memory management such as Python, or functional programming languages such as Haskell and OCaml.

In some parts of the implementation, for instance in the definition of the
sets $\CD,\CG,\CR$ and $\CI$, additional auxiliary functions were introduced
in order to make the code shorter and more readable. This made the final
implementation considerably clearer. In particular, together with suitable
comments, the corresponding proof constructions became much easier to check.

On the other hand, this design also led to the use of \texttt{clone} in
more places than would probably be necessary in a fully optimized
implementation. This may have a negative effect on performance, especially
when large recursive structures are copied. For the relatively small
applications considered here, this was not a serious problem.

It is conceivable that a more careful choice of smart pointers, including
alternatives to \texttt{Box} such as \texttt{Rc} or \texttt{Arc}, possibly
combined with parallel computation, could lead to a significantly more
efficient implementation. Investigating such optimizations, however, was
not the goal of the present implementation.

Overall, these features turned out to be a mixed blessing. In the present
implementation, which is only a relatively small prover, the
ownership system and the borrow checker often made the code more verbose.
Thus, they were sometimes more of an obstacle than an advantage. For
larger proof assistants, however, where performance and memory safety are
more important, precisely these features may become a significant
advantage.

\subsection{Further Noteworthy Features of Rust}
Finally, we mention two further features of Rust that were very useful in
the implementation, presumably not only in the development of a proof
assistant.

\subsubsection{Compile-Time Safety and Compiler Support}
The type safety of Rust, which is checked at compile time, made the
implementation considerably more pleasant. For example, when pattern
matching on an \texttt{enum}, the compiler checks whether all cases have
been covered. The same also applies to handling the data types \texttt{Option} and \texttt{Result}. Another advantage is that different syntactic
categories can be represented by different Rust types. Thus, types, terms,
formulas, and proofs are kept separate by the programming language itself,
which prevents many accidental misuses. Moreover, the error messages of the Rust compiler are usually very informative.
With a suitable editor, they can be displayed even more clearly.

\subsubsection{Modularity and Testing}
Another useful aspect of Rust is its module system. It makes it possible
to organize the implementation into several files and directories, for
example by putting the definitions of types, terms, formulas, and proofs
into separate modules. This modular structure is particularly useful for a
proof assistant, since the different syntactic categories and operations
can be developed and checked independently.

Closely related to this is Rust's good support for testing. Tests can
be placed directly inside the corresponding source files, usually in a
separate test module. This was very helpful in the present implementation.
This is particularly useful for a proof
assistant, since constructors, substitutions, proof rules, and other basic
operations must be tested systematically.

\bibstyle{plain}
\begingroup
\emergencystretch=1em
\printbibliography
\endgroup
\end{document}